\documentclass{amsart}
\usepackage{amssymb}

\title{Arc-transitive graphs of valency $8$ have a semiregular automorphism}

\author[G. Verret]{Gabriel Verret}
\address{Gabriel Verret, Centre for Mathematics of Symmetry and Computation, \newline
\indent School of Mathematics and Statistics, The University of Western Australia, \newline
\indent 35 Stirling Highway, Crawley, WA 6009, Australia. \newline
\indent Also affiliated with : UP FAMNIT, University of Primorska, \newline
\indent Glagolja\v{s}ka 8, 6000 Koper, Slovenia.}

\email{gabriel.verret@uwa.edu.au}

\newtheorem{theorem}{Theorem}[section]
\newtheorem{lemma}[theorem]{Lemma}

\newcommand{\V}{{\rm{V}}}
\newcommand{\A}{{\rm{A}}}
\newcommand{\Aut}{{\rm{Aut}}}
\newcommand{\Sym}{{\rm Sym}}

\newcommand{\norml}{\trianglelefteq}

\newcommand{\AlG}{{\rm Al}}
\newcommand{\vGa}{\vec{\Gamma}}

\thanks{The author is supported by UWA as part of the Australian Research Council grant DE130101001.}


\begin{document}
\begin{abstract}
One version of the polycirculant conjecture states that every vertex-transitive graph has a semiregular automorphism.  We give a proof of the conjecture in the arc-transitive case for graphs of valency $8$, which was the smallest open case.
\end{abstract}

\maketitle

\section{Introduction}
All graphs considered in this paper are finite and simple. An automorphism of a graph $\Gamma$ is a permutation of the vertex-set $\V(\Gamma)$ which preserves the adjacency relation. The set of automorphisms of $\Gamma$ form a group, denoted $\Aut(\Gamma)$. A graph $\Gamma$ is said to be $G$-\emph{vertex-transitive} if $G$ is a subgroup of $\Aut(\Gamma)$ acting transitively on $\V(\Gamma)$. Similarly, $\Gamma$ is said to be $G$-\emph{arc-transitive} if $G$ acts transitively on the arcs of $\Gamma$. (An \emph{arc} is an ordered pair of adjacent vertices). When $G=\Aut(\Gamma)$, the prefix $G$ in the above notation is sometimes omitted. 

A permutation group is called \emph{semiregular} if it has no non-identity element which fixes a point. A permutation is called semiregular if it generates a semiregular group. Equivalently, a permutation is semiregular if all of its cycles have the same length.

In 1981, Maru\v{s}i\v{c} conjectured~\cite{Dragan} that every vertex-transitive graph with at least $2$ vertices has a non-identity semiregular automorphism. This is sometimes called the polycirculant conjecture. While there has been a lot of work on this conjecture and some of its variants, it is still wide open. There has been progress in some directions. For example, the conjecture has been settled for graphs of valency at most $4$~\cite{Dobson,DraganScap}, while the valency $5$ case is still open.

In the arc-transitive case, slightly more can be said, but we first need a few definitions. If $\Gamma$ is  a graph and $G\in\Aut(\Gamma)$, then $G_v^{\Gamma(v)}$ denotes the permutation group induced by the action of the vertex-stabiliser $G_v$ on the neighbourhood $\Gamma(v)$ of the vertex $v$. A permutation group is called \emph{quasiprimitive} if each of its nontrivial normal subgroup is transitive, while a  vertex-transitive graph $\Gamma$ is called \emph{locally-quasiprimitive} if  $\Aut(\Gamma)_v^{\Gamma(v)}$ is quasiprimitive. Using this terminology, we have the following theorem due to Giudici and Xu. 

\begin{theorem}\cite[Theorem~1.1]{MichaelXu}\label{Theorem:MichaelXu}
Let $\Gamma$ be a locally-quasiprimitive vertex-transitive graph. Then $\Gamma$ has a non-identity semiregular automorphism.
\end{theorem}

A transitive group of prime degree is clearly quasiprimitive (in fact primitive) and hence it follows from Theorem~\ref{Theorem:MichaelXu} that arc-transitive graphs of prime valency have a non-identity semiregular automorphism. In an upcoming paper~\cite{MichaelVerret}, we deal with the case of arc-transitive graphs of valency twice a prime. The main result of this paper is to deal with the smallest open valency.

\begin{theorem}\label{theo:main}
Let $\Gamma$ be an arc-transitive graph of valency $8$. Then $\Gamma$ has a non-identity semiregular automorphism.
\end{theorem}

\section{Preliminaries}
\subsection{Permutation groups}

We start with a few very basic lemmas about permutation groups.  Since the proofs are short, we include them for the sake of completeness.

\begin{lemma}\label{lemma:primepower}
Let $p$ be a prime and let $G$ be a transitive permutation group of degree a power of $p$. Then $G$ contains a non-identity semiregular element.
\end{lemma}
\begin{proof}
Let $S$ be a Sylow $p$-subgroup of $G$. By \cite[Theorem 3.4']{wielandt}, $S$ is transitive. Let $z$ be a non-identity element of the centre of $S$. If $z$ fixes a point, then it must fix all of them, which is a contradiction. It follows that $z$ is semiregular.
\end{proof}

\begin{lemma}\label{lemma:abelian2orbits}
If $G$ is a transitive permutation group with a non-trivial abelian normal subgroup $N$ that has at most two orbits, then $N$ contains a non-identity semiregular element.
\end{lemma}
\begin{proof}
If $N$ is semiregular, then we are done. We may thus assume that there exists a non-identity element $n\in N$ fixing some point $v$.  Since $N$ is abelian, $n$ fixes the $N$-orbit $v^N$ pointwise. It follows that $N$ has exactly two orbits. Let $u$ be a point not in $v^N$ and let $g\in G$ such that $v^g=u$. Then $n$ acts semiregularly on $u^N$, while $n^g$ fixes $u^N$ pointwise and acts semiregularly on $v^N$. It follows that $nn^g$ is a non-identity semiregular element of $N$.
\end{proof}

\begin{lemma}\label{lemma:coprime}
Let $G$ be a permutation group and let $K$ be a normal subgroup of $G$ such that $G/K$ acts faithfully on the $K$-orbits. If $G/K$ contains a semiregular element $gK$ of order coprime to $|K|$, then $G$ contains a semiregular element of order $|gK|$.
\end{lemma}
\begin{proof}
There exists $k\in K$ such that $g^{|gK|}=k$. Let $h=g^{|k|}$. Since $|k|$ and $|gK|$ are coprime, it follows that $|h|=|gK|$ and $hK$ is a non-identity power of $gK$ hence it is semiregular. We now show that $h$ is semiregular. Suppose $h^i$ fixes a point $v$ for some integer $i$. Then $h^iK=(hK)^i$ fixes $v^K$ and, since $hK$ is semiregular, $h^iK=K$. This implies that $h^i\in K$. Since $h$ has order coprime to $|K|$, it follows that $h^i=1$.
\end{proof}

\begin{lemma}\label{Lemma:AtMost9}
A transitive permutation group of degree $8$ is either primitive or it is a $\{2,3\}$-group.
\end{lemma}
\begin{proof}
Let $G$ be a transitive permutation group of degree $8$ that is not primitive. Then $G$ has a non-trivial block system, with blocks of size either $2$ or $4$. It follows that $G$ is isomorphic to a subgroup of $\Sym(4)\wr\Sym(2)$ or $\Sym(2)\wr\Sym(4)$ and hence it is a $\{2,3\}$-group.
\end{proof}

\subsection{Quotient graphs}

We will need some basic facts about quotient graphs, which we now collect. Let $\Gamma$ be a graph and let $N\leq\Aut(\Gamma)$. The {\em quotient graph} $\Gamma/N$ is the graph whose vertices are the $N$-orbits with  two such $N$-orbits $v^N$ and $u^N$ adjacent whenever there is a pair of vertices $v'\in v^N$ and $u'\in u^N$ that are adjacent in $\Gamma$.  Clearly, if $\Gamma$ is connected then so is $\Gamma/N$.

Let $\Gamma$ be a $G$-arc-transitive graph, let $N\norml G$ and let $K$ be the kernel of the action of $G$ on $N$-orbits. Then $K=NK_v\norml G$, $G/K\leq\Aut(\Gamma/N)$ and $\Gamma/N$ is $G/K$-arc-transitive. If $N$ has at least $3$ orbits, then the valency of $\Gamma/N$ is at least $2$ and divides the valency of $\Gamma$. If $\Gamma/N$ has the same valency as $\Gamma$, then $N=K$ and $K$ is semiregular.

Some of these facts are used in the proof of the following lemma.

\begin{lemma}\label{lemma:4AT}
Let $p$ be an odd prime and let $\Gamma$ be a $4$-valent $G$-arc-transitive graph such that $p$ divides $|\V(\Gamma)|$. If $G$ is solvable, then $G$ contains a semiregular element of order $p$.
\end{lemma}
\begin{proof}
The proof goes by induction on $|\V(\Gamma)|$. Clearly we may assume that $\Gamma$ is connected. Let $v\in\V(\Gamma)$. If $G_v^{\Gamma(v)}$ is a $2$-group then, by vertex-transitivity of $G$ and connectedness of $\Gamma$, it follows that $G_v$ is a $2$-group and every element of $G$ of order $p$ is semiregular. We thus assume that $G_v^{\Gamma(v)}$ is not a $2$-group and, since it is a transitive group of degree $4$,  it is $2$-transitive. Since $G$ is solvable, it contains a non-trivial normal elementary abelian $q$-group $N$. If $N$ has at most two orbits, then $p=q$ and the result follows from Lemma~\ref{lemma:abelian2orbits}. 

From now on, we assume that $N$ has at least three orbits. Since $G_v^{\Gamma(v)}$ is $2$-transitive, this implies that $\Gamma/N$ is $4$-valent and hence $N$ is semiregular and $G/N$ acts faithfully on $\Gamma/N$. If $p=q$, then $N$ contains a semiregular element of order $p$ and the conclusion holds. Suppose now that $p\neq q$. Then $\Gamma/N$ is a connected, $4$-valent $G/N$-arc-transitive graph. Since $p$ is coprime to $|N|$, it divides $|\V(\Gamma/N)|$. By the induction hypothesis, $G/N$ contains a semiregular element of order $p$. The result then follows from Lemma~\ref{lemma:coprime}.
\end{proof}

\subsection{Digraphs}
Finally, we will need a few notions about digraphs. We follow the terminology of~\cite{Digraph} closely. A {\em digraph} $\vGa$ consists of a finite non-empty set of \emph{vertices} $\V(\vGa)$ and a set of \emph{arcs} $\A(\vGa)\subseteq V\times V$, which is an arbitrary binary relation on $V$. A digraph $\vGa$ is called {\em asymmetric} provided that the relation $\A(\vGa)$ is asymmetric.

If $(u,v)$ is an arc of $\vGa$, then we say that $v$ is an {\em out-neighbour} of $u$ and that $u$ is an {\em in-neighbour} of $v$. The symbols $\vGa^+(v)$ and $\vGa^-(v)$ will denote the set of out-neighbors of $v$ and the set of in-neighbors of $v$, respectively. We also say that $u$ is the {\em tail} and $v$ the {\em head} of $(u,v)$, respectively.  The digraph $\vGa$ is said to be of out-valence $k$ if $|\vGa^+(v)| = k$ for every $v\in \V(\vGa)$. 

An automorphism of $\vGa$ is a permutation of $\V(\vGa)$ which preserves the relation $\A(\vGa)$. The set of automorphisms of $\vGa$ form a group, denoted $\Aut(\vGa)$. We say that $\vGa$ is arc-transitive provided that $\Aut(\vGa)$ acts transitively on $\A(\vGa)$. 

We say that two arcs $a$ and $b$ of $\vGa$ are {\em related} if they have a common tail or a common head. Let $R$ denote the transitive closure of this relation. The {\em alternet} of $\vGa$ (with respect to $a$) is the subdigraph of $\vGa$ induced by the $R$-equivalence class $R(a)$ of the arc $a$.  (i.e. the digraph with vertex-set consisting of all heads and tails of arcs in $R(a)$ and whose arc-set is $R(a)$). If the alternet with respect to $(u,v)$ contains an arc of the form $(v,w)$, then this alternet is called \emph{degenerate}. 

If the alternet of the arc $(u,v)$ is non-degenerate, then it is a connected bipartite digraph where the first bipartition set consists only of sources while the second bipartition set contains only sinks. An important case occurs when this alternet is in fact a complete bipartite digraph in which case we will simply say that the alternet is {\em complete bipartite}. We say that $\vGa$ is {\em loosely attached} if $\vGa$ has no degenerate alternets and the intersection of the set of sinks of one alternet intersects the set of sources of another alternet in at most one vertex. 

We define the {\em digraph of alternets} $\AlG(\vGa)$  of $\vGa$ as the digraph the vertices of which are the alternets of $\vGa$ and with two alternets $A$ and $B$ forming an arc $(A, B)$ of $\AlG(\vGa)$ whenever the intersection of the set of sinks of $A$ with the set of sources of $B$ is non-empty. 

We are now ready to prove the following theorem.

\begin{theorem}\label{theo:digraph}
Let $\vGa$ be a connected asymmetric arc-transitive digraph of out-valence $4$. Then $\vGa$ has a non-identity semiregular automorphism.
\end{theorem}
\begin{proof}
The proof goes by induction on $|\V(\vGa)|$. Since $\vGa$ is connected and arc-transitive (and finite), it follows that it is vertex-transitive and strongly connected (see for example~\cite[Lemma 2]{Neu}).

Let $v\in\V(\vGa)$ and let $G=\Aut(\vGa)$. Without loss of generality, we may assume that every prime that divides $|G|$ also divides $|G_v|$. If $G_v$ is a $2$-group, then the conclusion follows from Lemma~\ref{lemma:primepower}. We may thus assume that $G_v$ is not a $2$-group and hence neither is $G_v^{\Gamma^+(v)}$. Since $G_v^{\Gamma^+(v)}$ is a transitive permutation group of degree $4$, it is $2$-transitive. 

Since $\vGa$ has out-valence $4$, $G_v^{\Gamma^+(v)}$ is a $\{2,3\}$-group, hence so are $G_v$ and $G$ and therefore $G$ is solvable by Burnside's Theorem. It follows that $G$ has an abelian minimal normal subgroup $N$. If $N$ is semiregular, the conclusion holds. We may thus assume that $N$ is not semiregular and hence $N_v^{\Gamma^+(v)}$ is a non-trivial normal subgroup of $G_v^{\Gamma^+(v)}$ and therefore is transitive. The same argument yields that $N_v^{\Gamma^-(v)}$ is also transitive.

Let $(u,v)$ be an arc of $\vGa$. We have just seen that $N_u^{\Gamma^+(u)}$ and $N_v^{\Gamma^-(v)}$ are both transitive. This implies that $\Gamma^+(u)\subseteq v^N$ and $\Gamma^-(v)\subseteq u^N$. On the other hand, $N$ is abelian and hence $N_u$ fixes $u^N$ pointwise and $N_v$ fixes $v^N$ pointwise. It follows that the alternet of $\vGa$ with respect to $(u,v)$ is not degenerate and is complete bipartite.

If $\vGa$ is not loosely attached, then it follows that, for every vertex $x$, there exists at least one other vertex $y$ such that $\vGa^-(x)=\vGa^-(y)$ and $\vGa^+(x)=\vGa^+(y)$. It follows easily that $\vGa$ has a non-identity semiregular automorphism in this case.

We may thus assume that $\vGa$ is loosely attached. Let $\vGa'=\AlG(\vGa)$ be the digraph of alternet of $\vGa$. It follows from~\cite[Lemmas~3.1-3.3]{Digraph} that $\vGa'$ is a connected asymmetric digraph of out-valence $4$ and that $G=\Aut(\vGa')$. It also follows easily that an automorphism which is semiregular on $\vGa'$ is also semiregular on $\vGa$.

Note that $|\V(\vGa')|=|\V(\vGa)|/4$ and hence we may apply the induction hypothesis to $\vGa'$ to conclude that it has a non-identity semiregular automorphism $g$. Together with, the observation in the previous sentence, this concludes the proof.
\end{proof}

\section{Proof of Theorem~\ref{theo:main}}
Let $\Gamma$ be an arc-transitive graph of valency $8$. We must show that $\Gamma$ has a non-identity semiregular automorphism.

Clearly, we may assume that $\Gamma$ is connected. Let $G=\Aut(\Gamma)$. If $\Gamma$ is locally-quasiprimitive, the result follows from~Theorem~\ref{Theorem:MichaelXu}. We therefore assume that $G_v^{\Gamma(v)}$ is not quasiprimitive. By Lemma~\ref{Lemma:AtMost9}, $G_v^{\Gamma(v)}$ is a $\{2,3\}$-group and hence so is $G_v$. We may also assume that $G$ itself is a $\{2,3\}$-group and hence it is solvable by Burnside's Theorem.

It follows that $G$ has an elementary abelian minimal normal subgroup $N$. If $N$ has at most two orbits,  the result follows from Lemma~\ref{lemma:abelian2orbits}. We may thus assume that $N$ has at least three orbits and, in particular, $N_v^{\Gamma(v)}$ is intransitive. If $N$ is semiregular, then the conclusion follows. We may thus assume that $N_v\neq 1$ and hence $N_v^{\Gamma(v)}\neq 1$.  In particular, $N_v^{\Gamma(v)}$ is a non-trivial, intransitive normal subgroup of $G_v^{\Gamma(v)}$. It follows that $N_v^{\Gamma(v)}$ is a $2$-group and hence $N$ is an elementary abelian $2$-group. Let $K$ be the kernel of the action of $G$ on $N$-orbits.

If $\Gamma/N$ is $4$-valent, then $v$ is adjacent to at most two vertices from any $N$-orbit. It follows that $K_v^{\Gamma(v)}$ is a $2$-group and hence so are $K_v$ and $K$. If $|\V(\Gamma/N)|$ is a power of $2$, then so is $|\V(\Gamma)|$ and the result follows from Lemma~\ref{lemma:primepower}. We may thus assume that $|\V(\Gamma/N)|$ is not a power of $2$ and, by Lemma~\ref{lemma:4AT}, $G/K$ contains a semiregular element of odd order. It then follows from Lemma~\ref{lemma:coprime} that $G$ contains a semiregular element of odd order.

It remains to deal with the case when $\Gamma/N$ is $2$-valent. In this case, there is a natural orientation of $\Gamma$ as a connected asymmetric $4$-valent digraph $\vGa$ and $\Aut(\vGa)$ is a subgroup of index $2$ in $G$. By Theorem~\ref{theo:digraph}, $\Aut(\vGa)$ has a semiregular element. This concludes the proof.

%%%%%%%%%%%%%%%%%%%%%%%%%%%%%%%%%%%%%%%%%%%%%%%%%%%%%%%%%%%%%%%%%%%%%%%%%%%%%%%
\bibliographystyle{amsplain}

\end{document}